\documentclass[12pt]{amsart}
\pdfoutput=1
\usepackage{amsmath, amssymb}
\usepackage{amsfonts}
\usepackage{mathrsfs}
\usepackage[arrow,matrix,curve,cmtip,ps]{xy}
\usepackage{graphicx}
\usepackage{amsthm}
\usepackage{float}
\usepackage{amsthm}
\usepackage[utf8]{inputenc}
\usepackage[T1]{fontenc}
\usepackage{mathtools}

\allowdisplaybreaks

\newtheorem{theorem}{Theorem}[section]
\newtheorem{lemma}[theorem]{Lemma}

\newtheorem{corollary}[theorem]{Corollary}

\newtheorem*{theorem*}{Theorem}
\theoremstyle{remark}
\newtheorem{remark}[theorem]{Remark}

\numberwithin{equation}{section}

\begin{document}
\title[ Quadratic Bounds on Quasiconvexity]{Quadratic Bounds on the Quasiconvexity of Nested Train Track Sequences}

\author{Tarik Aougab}

\address{Department of Mathematics \\ Yale University \\ 10 Hillhouse Avenue, New Haven, CT 06510 \\ USA}
\email{tarik.aougab@yale.edu}

\date{\today}


\begin{abstract}

Let $S_{g,p}$ denote the genus $g$ orientable surface with $p$ punctures. We show that nested train track sequences constitute $O((g,p)^{2})$-quasiconvex subsets of the curve graph, effectivizing a theorem of Masur and Minsky. As a consequence, the genus $g$ disk set is $O(g^{2})$-quasiconvex. We also show that splitting and sliding sequences of birecurrent train tracks project to $O((g,p)^{2})$-unparameterized quasi-geodesics in the curve graph of any essential subsurface, an effective version of a theorem of Masur, Mosher, and Schleimer.

\end{abstract}

\maketitle

\section{Introduction}
Let $S_{g,p}$ denote the orientable surface of genus $g$ with $p\geq 0$ punctures, and let $\mathcal{C}(S_{g,p})$ be the corresponding curve complex. Finally, let $\mathcal{C}_{k}(S_{g,p})$ denote the corresponding $k$-skeleton. 

Let $(\tau_{i})_{i}$ be a sequence of train tracks on $S_{g,p}$ such that $\tau_{i+1}$ is carried by $\tau_{i}$ for each $i$. Such a collection of train tracks defines a subset of $\mathcal{C}_{0}(S_{g,p})$ called a \textit{nested train track sequence}. A \textit{train track splitting sequence} is an important special case of such a sequence, in which $\tau_{i}$ is obtained from $\tau_{i-1}$ via one of two simple combinatorial moves, \textit{splitting} and \textit{sliding}. 

A nested train track sequence is said to have $R$-bounded steps if the $\mathcal{C}_{1}$-distance between the vertex cycles of $\tau_{i}$ and those of $\tau_{i+1}$ is bounded above by $R$.  Masur-Minsky \cite{Mas-MurIII} showed that any nested train track sequence with $R$-bounded steps is a $K=K(R,g,p)$-quasigeodesic. Our first result provides some effective control on $K$ as a function of $g$ and $p$:

\begin{theorem} Let $\omega(g,p)= 3g+p-4$. There exists a function $K(g,p)=O(\omega(g,p)^{2})$ such that any nested train track sequence with $R$-bounded steps is a $(K(g,p)+R)$-unparameterized quasi-geodesic of the curve graph $\mathcal{C}_{1}(S_{g,p})$, which is $(K(g,p)+R)$-quasiconvex.
\end{theorem}

Masur-Mosher-Schleimer \cite{Mas-Mos-Sch} used Masur and Minsky's result to show that if $Y\subseteq S_{g,p}$ is any essential subsurface, then a sliding and splitting sequence on $S_{g,p}$ maps to a uniform unparameterized quasi-geodesic under the subsurface projection map to $\mathcal{C}(Y)$. Using Theorem $1.1$,  we show:

\begin{theorem} There exists a function $A(g,p)= O(\omega(g,p)^{2})$ satisfying the following. Suppose $Y\subseteq S_{g,p}$ is an essential subsurface, and let $(\tau_{i})_{i}$ be a splitting and sliding sequence of birecurrent train tracks on $S_{g,p}$. Then $(\tau_{i})_{i}$ projects to an $A(g,p)$-unparameterized quasi-geodesic in $\mathcal{C}_{1}(Y)$. 
\end{theorem}

Let $H_{g}$ denote the genus $g$ handlebody and let $D(g)\subset \mathcal{C}_{1}(S_{g})$ denote the set of \textit{meridians}, curves on $S_{g}$ that bound disks in $H_{g}$. Also due to Masur and Minsky \cite{Mas-MurIII} is the fact that any two meridians in $D(g)$ can be connected by a $15$-bounded nested train track sequence. Therefore, as a corollary of Theorem $1.1$, we obtain:

\begin{corollary} There exists a function $f(g)= O(g^{2})$ such that $D(g)$ is an $f(g)$-quasiconvex subset of $\mathcal{C}_{1}(S_{g})$. 

\end{corollary}

The \textit{mapping class group}, denoted $\mbox{Mod}(S)$, is the group of isotopy classes of orientation preserving homeomorphisms of a surface $S$ (see \cite{Far-Mar} for a thorough exposition). 

As an application of Corollary $1.3$, we obtain a more effective approach for detecting when a pseudo-Anosov mapping class $\phi$ is generic. Here, \textit{generic} means that the stable lamination of $\phi$ is not a limit of meridians; the term generic is warranted by a theorem of Kerckhoff \cite{Ker}, which states that the set of all projective measured laminations which \textit{are} limits of meridians constitutes a measure $0$ subset of $\mathcal{PML}(S)$, the space of all projective measured laminations on a surface $S$. 

In what follows, let $d_{\mathcal{C}(S)}$ denote distance in $\mathcal{C}_{1}(S)$; when there is no confusion, the reference to $S$ will be omitted.  Masur and Minsky \cite{Mas-Mur} showed that $\mathcal{C}_{1}(S)$ is a $\delta$-hyperbolic metric space. 

Using Corollary $1.2$,  work of Abrams-Schleimer \cite{Ab-Sc}, and the fact that the curve graphs are uniformly hyperbolic (shown by the author in \cite{Aoug}, and independently by Bowditch \cite{BowII}, Clay-Rafi-Schleimer \cite{Cl-Ra-Sc}, and Hensel-Przytycky-Webb \cite{Hen-Prz-We}), we have:

\begin{corollary}
There exists a function $r(g)=O(g^{2})$ such that $\phi \in \mbox{Mod}(S_{g})$ is a generic pseudo-Anosov mapping class if and only if there exists some $k \in \mathbb{N}$ such that for all $n>k$, 
\[ d_{\mathcal{C}}(D(g),\phi^{n}(D(g)))> r(g). \]

\end{corollary} 

\begin{remark} By the argument of Abrams-Schleimer, it suffices to take $r(g)= 2\delta+2f(g)$, for $\delta$ the hyperbolicity constant of $\mathcal{C}_{1}$, and $f(g)$ as in the statement of Corollary $1.3$. $\Box$
\end{remark}

We also note that quasiconvexity of $D(g)$ and the fact that splitting sequences map to quasi-geodesics under subsurface projection are main ingredients in the proof due to Masur and Schleimer \cite{Mas-Sch} that the disk complex is $\delta$-hyperbolic. Thus, the effective control discussed above is perhaps a first step to studying the growth of the hyperbolicity constant of the disk complex.

\textbf{How the main theorem is proved.}
The proof of Theorem $1.1$ relies on the ability to control

\begin{enumerate}

\item the hyperbolicity constant $\delta(g,p)$ of $\mathcal{C}_{1}$;
\item $B= B(g,p)$, a bound on the diameter of a set of vertex cycles of a fixed train track $\tau \subset S_{g,p}$; and 
\item the ``nesting lemma constant'' $k(g,p)$. 

\end{enumerate}

As mentioned above, due to work of the author, Bowditch, Clay-Rafi-Schleimer and Hensel-Przytycky-Webb, curve graphs are uniformly hyperbolic. Furthermore, Hensel-Przytycky-Webb \cite{Hen-Prz-We} show that all curve graphs are $17$-hyperbolic. 

Regarding $(2)$, The author \cite{Aoug} has also shown that for sufficiently large $\omega, B(g,p)\leq 3$. 

Therefore, all that remains is to analyze the growth of $k(g,p)$, which we address in section $5$ by following Masur and Minsky's original argument while keeping track of the constants that pop up along the way. However, in order to do this, we have need of an effective criterion for determining when a train track $\tau$ is non-recurrent, which we address in section $4$.

\textbf{Organization of paper.} In section $2$, we review some preliminaries about curve complexes and subsurface projections. In section $3$, we review train tracks on surfaces and bounds on curve graph distance given by intersection number, as obtained in previous work. In section $4$, we obtain an effective way of detecting non-recurrence of train tracks by analyzing the linear algebra of the corresponding branch-switch incidence matrix. In section $5$, we obtain an effective version of Masur and Minsky's nesting lemma, which is the main tool needed to prove Theorem $1.1$. In section $6$ we complete the proofs of Theorems $1.1$, $1.2$, and Corollary $1.3$. 

\textbf{Acknowledgements.} The author would primarily like to thank his adviser, Yair Minsky, for his guidance and for many helpful suggestions. He would also like to thank Ian Biringer, Catherine Pfaff, Saul Schleimer, and Harold Sultan for their time and for the many motivating conversations they've had with the author regarding this work. 

\section{Preliminaries: Coarse Geometry, Combinatorial Complexes and Subsurface Projections}
Let $(X,d_{X})$, $(Y,d_{Y})$ be metric spaces. For some $k\geq 1$,  a relation $f:X\rightarrow Y$ is a $k$-\textit{quasi-isometric embedding } of $X$ into $Y$ if for any $x_{1},x_{2}\in X$ we have 
\[ d_{X}(x_{1},x_{2})\leq kd_{Y}(f(x_{1}),f(x_{2}))+k. \]
Since $f$ is not necessarily a map, $f(x), f(y)$ need not be singletons, and the distance $d_{Y}(f(x),f(y))$ is defined to be the diameter in the metric $d_{Y}$ of the union $f(x)\cup f(y)$. If the $k$-neighborhood of $f(X)$ is all of $Y$, then $f$ is a $k$-\textit{quasi-isometry} between $X$ and $Y$, and we refer to $X$ and $Y$ as being \textit{quasi-isometric}. 

Given an interval $\left[a,b\right]\in \mathbb{Z}$, a $k$-\textit{quasi-geodesic} in $X$ is a $k$-quasi-isometric embedding $f:\left[a,b\right] \rightarrow X$. If $f:\left[a,b\right]\rightarrow X$ is any relation such that there exists an interval $\left[c,d\right]$ and a strictly increasing function $g:\left[c,d\right] \rightarrow \left[a,b\right]$ such that $f\circ g$ is a $k$-quasigeodesic, we say that $f$ is a $k$-\textit{unparameterized quasi-geodesic}. In this case we also require that for each $i\in \left[c,d-1\right]$, the diameter of $f\left(\left[ g(i),g(i+1)  \right]\right)$ is at most $k$. We will sometimes refer to a quasi-geodesic by its image in the metric space $X$. 

A simple closed curve on $S_{g,p}$ is $\textit{essential}$ if it is homotopically non-trivial, and not homotopic into a neighborhood of a puncture.

The \textit{curve complex} of $S_{g,p}$, denoted $\mathcal{C}(S_{g,p})$, is the simplicial complex whose vertices correspond to isotopy classes of essential simple closed curves on $S_{g,p}$, and such that $k+1$ vertices span a $k$-simplex exactly when the corresponding $k+1$ isotopy classes can be realized disjointly on $S_{g,p}$. $\mathcal{C}(S)$ is made into a metric space by identifying each simplex with the standard Euclidean simplex with unit length edges.  Let $\mathcal{C}_{k}(S)$ denote the $k$-skeleton of $\mathcal{C}(S)$.

$\mathcal{C}(S)$ is a locally infinite, infinite diameter metric space. By a theorem of Masur and Minsky \cite{Mas-Mur}, $\mathcal{C}(S)$ is $\delta$-$\textit{hyperbolic}$ for some $\delta=\delta(S)>0$, meaning that the $\delta$-neighborhood of the union of any two edges of a geodesic triangle contains the third edge. 

$\mathcal{C}(S)$ admits an isometric (but not properly discontinuous) action of $\mbox{Mod}(S)$, and it is a flag complex, so that its combinatorics are completely encoded by $\mathcal{C}_{1}(S)$, the $\textit{curve graph}$; note also that $\mathcal{C}(S)$ is quasi-isometric to $\mathcal{C}_{1}(S)$, and therefore to study the coarse geometry of $\mathcal{C}$ it suffices to consider the curve graph. Let $d_{\mathcal{C}}$ denote distance in the curve graph. 

If $p\neq 0$, we can consider more general combinatorial complexes, which also allow vertices to represent essential arcs connecting punctures, up to isotopy. As such, define $\mathcal{AC}(S)$, the \textit{arc and curve complex of $S$} to be the simplicial complex whose vertices correspond to isotopy classes of essential simple closed curves and arcs on $S$. As with $\mathcal{C}(S)$, two vertices are connected by an edge if and only if the corresponding isotopy classes can be realized disjointly, and the higher dimensional skeleta are defined by requiring $\mathcal{AC}(S)$ to be flag. As with $\mathcal{C}$, denote by $\mathcal{AC}_{k}(S)$ the $k$-skeleton of $\mathcal{AC}(S)$. It is worth noting the $\mathcal{AC}(S)$ is quasi-isometric to $\mathcal{C}(S)$, with quasi-constants not depending on the topological type of $S$. 

Let $Y \subseteq S$ be an essential, embedded subsurface of $S$ which is not a peripheral annulus. Then there is a covering space $S^{Y}$ associated to the inclusion $\pi_{1}(Y)<\pi_{1}(S)$. While $S^{Y}$ is not-compact, note that the Gromov compactification of $S^{Y}$ is homeomorphic to $Y$, and via this homeomorphism we identify $\mathcal{AC}(Y)$ with $\mathcal{AC}\left(S^{Y}\right)$. Then, given $\alpha \in \mathcal{AC}_{0}(S)$, the \textit{subsurface projection map} $\pi_{Y}:\mathcal{AC}(S)\rightarrow \mathcal{AC}(Y)$ is defined by setting $\pi_{Y}(\alpha)$ equal to its preimage under the covering map $S^{Y}\rightarrow S$. 

Technically, this defines a map from $\mathcal{AC}_{0}(S)$ into $2^{\mathcal{AC}_{0}(Y)}$ since their may be multiple connected components of the pre-image of a curve or arc , but the image of any point in the domain is a bounded subset of the range. Thus to make $\pi_{Y}$ a map we can simply choose some component of this pre-image for each point in the domain, and then extend the map $\pi_{Y}$ simplicially to the higher dimensional skeleta. 

Given an arc $a\in \mathcal{AC}(S)$, there is a closely related simple closed curve $\tau(a) \in \mathcal{C}_{1}(S)$, obtained from $a$ by surgering along the boundary components that $a$ meets. More concretely, let $\mathcal{N}(a)$ denote a thickening of the union of $a$ together with the (at most two) boundary components of $S$ that $a$ meets, and define $\tau(a)\in 2^{\mathcal{C}_{1}(S)}$ to be the components of $\partial(N(a))$. 

Thus we obtain a \textit{subsurface projection} map 
\[\psi_{Y}:= \tau \circ \pi_{Y}: \mathcal{C}(S)\rightarrow \mathcal{C}(Y)\]
 for $Y\subseteq S$ any essential subsurface. Here, a subsurface is \textit{essential} if it is not a thrice punctured sphere or an annulus whose core curve is homotopic into a neighborhood of a puncture of $S$.

Then given $\alpha,\beta \in \mathcal{C}(S)$, define $d_{Y}(\alpha,\beta)$ by

\[ d_{Y}(\alpha,\beta):= \mbox{diam}_{\mathcal{C}(Y)}(\psi_{Y}(\alpha)\cup \psi_{Y}(\beta)).\]

\section{ Train tracks and Intersection Numbers}
In this section, we recall some basic terminology of train tracks on surfaces; we refer to Penner-Harer \cite{Pen-Har} and Mosher \cite{Mos} for a more in-depth discussion. A \textit{train track} $\tau \subset S$ is an embedded $1$-complex whose vertices and edges are called \textit{switches} and \textit{branches}, respectively. Branches are smooth parameterized paths with well-defined tangent vectors at the initial and terminal switches. At each switch $v$ there is a unique line $L\subset T_{v}S$ such that the tangent vector of any branch incident at $v$ coincides with $L$. 

As part of the data of $\tau$, we choose a preferred direction along this line at each switch $v$; a half branch incident at $v$ is called \textit{incoming} if its tangent vector at $v$ is parallel to this chosen direction, and is called \textit{outgoing} if it is anti-parallel. Therefore at each switch, the incident half branches are partitioned disjointly into two orientation classes, the \textit{incoming germ} and \textit{outgoing germ}.

 The valence of each switch must be at least $3$ unless $\tau$ has a connected component consisting of a simple closed curve; in this case $\tau$ has one bivalent switch for such a component. 

Finally, we require that every complementary component of $S\setminus \tau$ has a negative generalized Euler characteristic, that is 
\[ \chi(Q)- \frac{1}{2}V(Q) <0\]
for any complementary component $Q$; here $\chi(Q)$ is the usual Euler characteristic and $V(Q)$ is the number of cusps on $\partial (Q)$.

A \textit{train path} is a path $\gamma :[0,1] \rightarrow \tau$, smooth on $(0,1)$, which traverses a switch only by entering via one germ and exiting from the other; a \textit{closed train path} is a train path  with $\gamma(0)=\gamma(1)$. A \textit{proper closed train path} is a closed train path with $\gamma'(0)= \gamma'(1)$; here $\gamma'(t)$ is the unit tangent vector to the path $\gamma$ at time $t$.

 Let $\mathcal{B}$ denote the set of branches of $\tau$; then a non-negative, real-valued function $\mu:\mathcal{B}\rightarrow \mathbb{R}_{+}$ is called a \textit{transverse measure} on $\tau$ if for each switch $v$ of $\tau$, we have 
\[ \sum_{b\in i(v)}\mu(b) =\sum_{b'\in o(v)}\mu(b')  \]
where $i(v)$ is the set of incoming branches, and $o(v)$ the set of outgoing ones. These are called the \textit{switch conditions}. $\tau$ is called \textit{recurrent} if it admits a strictly positive transverse measure, that is, one that assigns a positive weight to every branch.  A switch of $\tau$ is called \textit{semi-generic} if exactly one of the two germs of half branches consists of a single half branch. $\tau$ is called semi-generic if all switches are semi-generic, and $\tau$ is \textit{generic} if $\tau$ is semi-generic and each switch has degree at most $3$.  $\tau$ is called \textit{large} if each connected component of its complement is simply connected. 

Any positive scaling of a transverse measure is also a transverse measure, and therefore the set of all transverse measures, viewed as a subset of $\mathbb{R}^{\mathcal{B}}$ is a cone over a compact polyhedron in projective space. Let $P(\tau)$ denote the projective polyhedron of transverse measures.  A projective measure class $[\mu]\in P(\tau)$ is called a \textit{vertex cycle} if it is an extreme point of $P(\tau)$. It is worth noting that if $\tau$ is any train track on $S$, there exists a generic, recurrent train track $\tau'$ such that $P(\tau)=P(\tau')$. 

A lamination $\lambda$ is \textit{carried} by $\tau$ if there is a smooth map $\phi:S\rightarrow S$ called the \textit{carrying map for} $\lambda$ which is isotopic to the identity, $\phi(\lambda)\subset \tau$, and such that the restriction of the differential $d\phi$ to any tangent line of $\lambda$ is non-singular. If $c$ is any simple closed curve carried by $\tau$, then $c$ induces an integral transverse measure called the \textit{counting measure}, for which each branch of $\tau$ is assigned the natural number equaling the number of times the image of $c$ under its carrying map traverses that branch.

A subset $\tau'\subset \tau$ is called a \textit{subtrack} of $\tau$ if it is also a train track on $S$. In this case, we write $\tau'<\tau$.

Given any train track $\tau$ with branch set $\mathcal{B}$, we can distinguish branches as being one of three types: if $b\in \mathcal{B}$ and both half branches of $b$ are the only half branch in their respective germs, $b$ is called \textit{ large}. If both half branches of $b$ are in germs containing more than one half branch, $b$ is \textit{ small}; otherwise $b$ is \textit{mixed} (Figure $(2)$).

\begin{figure}[H]
\centering
	\includegraphics[width=3.5in]{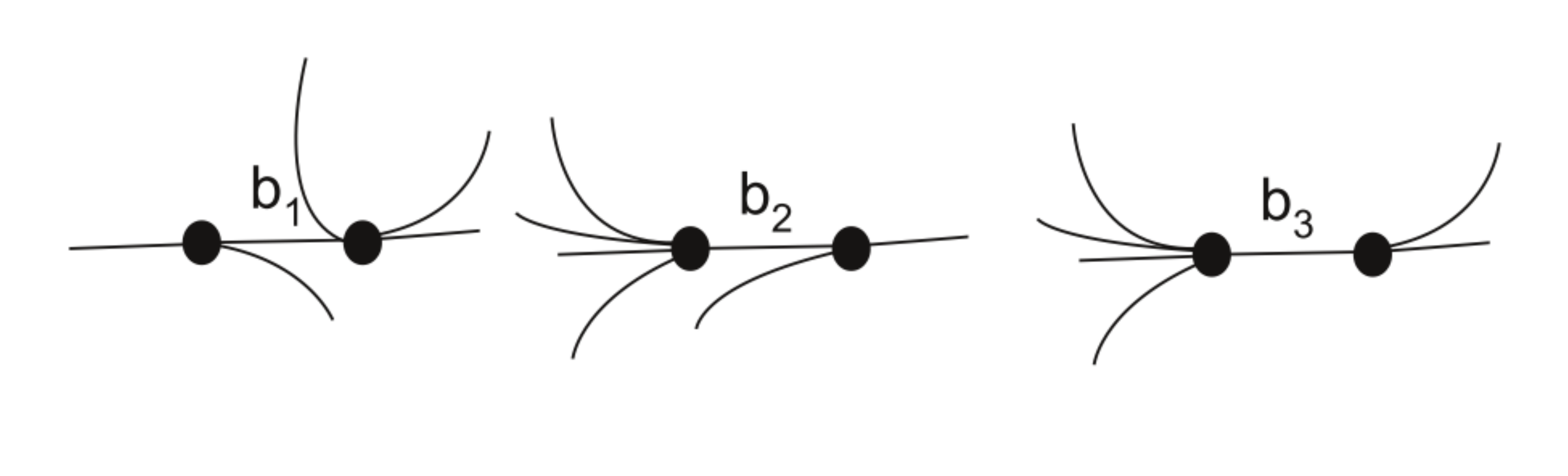}
\caption{ LEFT: $b_{1}$ is small; MIDDLE: $b_{2}$ is mixed; RIGHT: $b_{3}$ is large.}
\end{figure}

If $[v]$ is a vertex cycle of $\tau$, then there is a unique (up to isotopy) simple closed curve $c(v)$ such that $c$ is carried by $\tau$, and the counting measure on $c$ is an element of $[v]$. Therefore, if $[v_{1}],[v_{2}]$ are two vertex cycles of $\tau$, we can define the distance $d([v_{1}],[v_{2}])$ between them to be the curve graph distance between their respective simple closed curve representatives:

\[ d([v_{1}],[v_{2}]):= d_{\mathcal{C}}(c(v_{1}),c(v_{2})) \]

Using this, we can also define the distance between two train tracks $\tau$ and $\tau'$ to be the distance between their vertex cycle sets:
\[ d_{T}(\tau,\tau'):= \min \left\{d([v_{\tau}],[v_{\tau'}]): [v_{\tau}] \hspace{1 mm} \mbox{is a vertex cycle of} \hspace{1 mm} \tau \hspace{1 mm} \mbox{and} \hspace{1 mm} [v_{\tau'}] \hspace{1 mm} \mbox{is a vertex cycle of} \hspace{1 mm} \tau'  \right\} \]

A \textit{nested train track sequence} is a sequence $(\tau_{i})_{i}$ on $S_{g,p}$ of birecurrent train tracks such that $\tau_{j}$ is carried by $\tau_{j+1}$ for each $j$. This in turn determines a collection of vertices in $\mathcal{C}_{1}(S_{g,p})$, by associating the track $\tau_{j}$ with its collection of vertices. 

Given $R>0$, a nested train track sequence $(\tau_{i})_{i}$ is said to have $R$-\textit{bounded steps} if 
\[ d_{T}(\tau_{i},\tau_{i+1})\leq R \]
for each $i$. An important special case is the example of a \textit{splitting and sliding sequence}. This is any train track sequence where $\tau_{i}$ is obtained from $\tau_{i+1}$ via one of two combinatorial moves, \textit{splitting} (Figure $2$) and \textit{sliding} (Figure $3$). 

\begin{figure}[H]
\centering
	\includegraphics[width=3.5in]{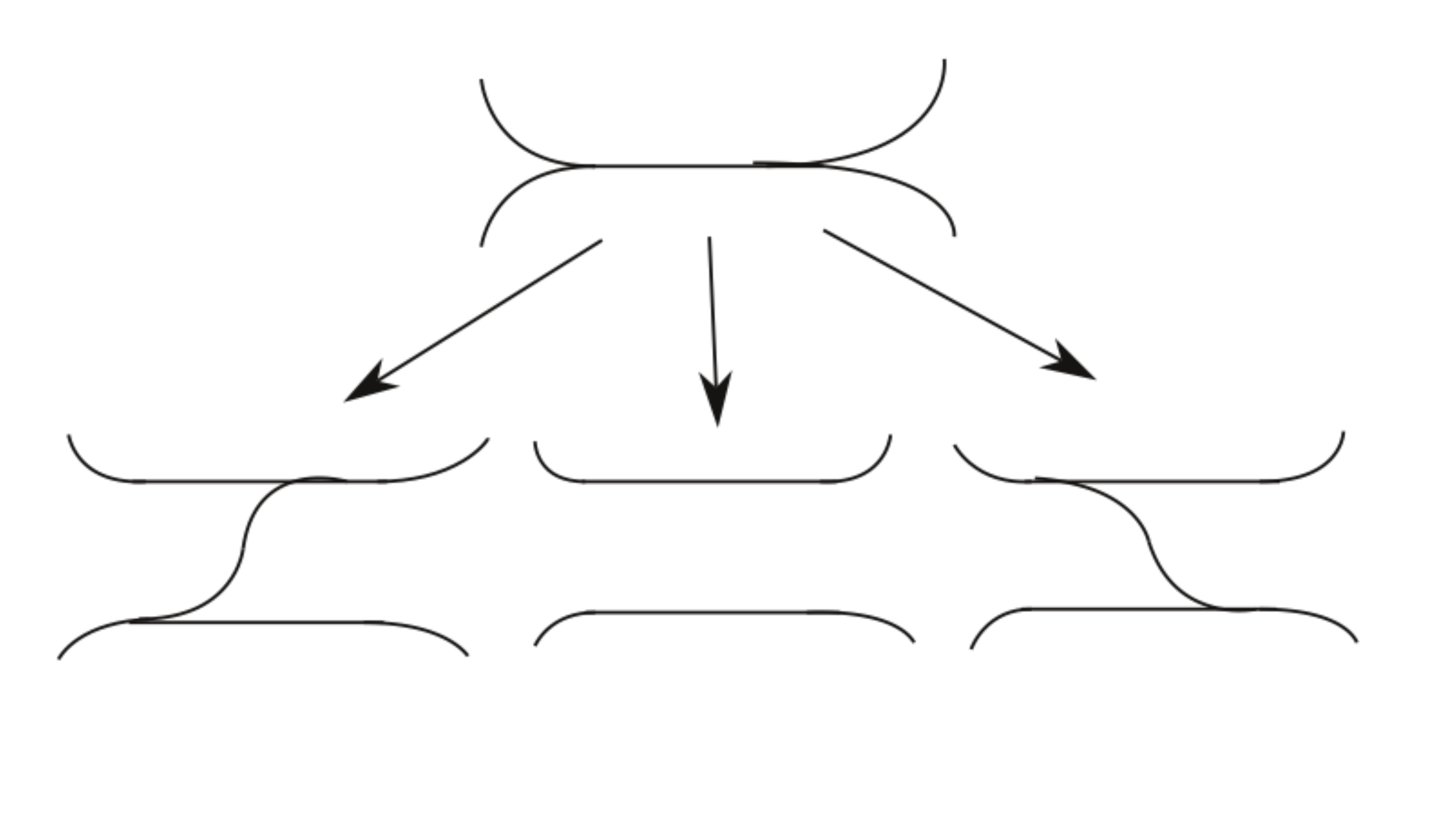}
\caption{  Any large branch admits three possible ``splittings''.}
\end{figure}

\begin{figure}[H]
\centering
	\includegraphics[width=3.5in]{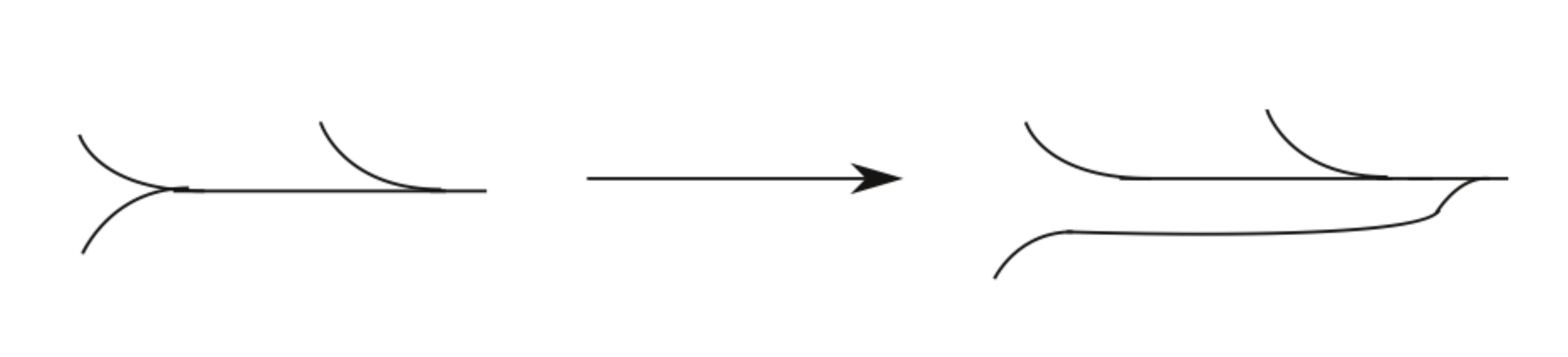}
\caption{ Any mixed branch admits a ``sliding'', as above.}
\end{figure} 

We will need the following theorem, as seen in previous work of the author \cite{Aoug}:

\begin{theorem} There exists a natural number $n\in \mathbb{N}$ such that if $\omega(g,p)>n$, the following holds: 
suppose $\tau\subset S_{g,p}$ is any train track and $[v_{1}],[v_{2}]$ are vertex cycles of $\tau$. Then 
\[ d([v_{1}],[v_{2}])\leq 3. \]

\end{theorem}

Let $int(P(\tau))\subset P(\tau)$ denote the set of strictly positive transverse measures on $\tau$. There $\tau$ is recurrent if and only if $int(P(\tau))\neq \emptyset$. For $\tau$ a large track, a \textit{diagonal extension} $\sigma$ of $\tau$ is a track such that $\tau<\sigma$ and and each branch of $\sigma\setminus \tau$ has the property that its endpoints are incident at corners of complementary regions of $\tau$. 

Following Masur and Minsky \cite{Mas-Mur}, let $E(\tau)$ denote the set of all diagonal extensions of $\tau$, and define 
\[ PE(\tau):= \bigcup_{\sigma \in E(\tau)}P(\sigma). \]

Let $N(\tau)$ be the union of $E(\kappa)$ over all large, recurrent subtracks $\kappa<\tau$:

\[ N(\tau):= \bigcup_{\kappa<\tau, \kappa \hspace{1 mm} large, recurrent}E(\kappa), \]

and define
\[ PN(\tau):= \bigcup_{\kappa\in N(\tau)}P(\kappa) \]

Define $int(PE(\tau))$ to be the measures in $PE(\tau)$ whose restrictions to $\tau$ are strictly positive, and define 
\[ int(PN(\tau)):= \bigcup_{\kappa} int(PE(\kappa)). \]

We conclude this section with the statement of a previous result of the author \cite{Aoug} which will be heavily relied upon in section $3$. 

\begin{theorem} For $\lambda \in (0,1)$, there is some $N=N(\lambda)$ such that if $\alpha,\beta\in \mathcal{C}_{0}(S_{g})$, whenever $\omega(g,p)>N(\lambda)$ and $d_{\mathcal{C}}(\alpha,\beta)\geq k$, 
\[ i(\alpha,\beta) \geq \left(\frac{(\omega(g,p)^{\lambda}}{q(g,p)}\right)^{k-2} \]
where $q(g,p)=O(\log_{2}(\omega))$. 
\end{theorem}

\begin{remark} In the above, $i(\alpha,\beta)$ is the \textit{geometric intersection number} between $\alpha$ and $\beta$, defined by 
\[ i(\alpha,\beta):= \min |x\cap \beta| \]
where the minimum is taken over all $x$ isotopic to $\alpha$. 
\end{remark}

We can explicitly write down the function $q(g,p)$ from the statement of Theorem $3.2$. $q(g,p)$ is an upper bound on the girth of a finite graph with at most $8(6g+3p-7)$ vertices, and average degree larger than $2.02$. As seen in Fiorini-Joret-Theis-Wood \cite{Fio-Jor-The-Woo},  
\[ q(g,p)= \left(\frac{8}{\log_{2}(1.01)}+5\right)\log_{2}(8(6g+3p-7))  \]
\[ < 1000\cdot \log_{2}(100\omega). \]
This upper bound will be used in section $5$.

\section{Detecting Recurrence from the Incidence Matrix }
Let $\tau=(\mathcal{S},\mathcal{B})\subset S_{g,p}$ be a train track with branch set $\mathcal{B}$ and switch set $\mathcal{S}$.

Label the branches $\mathcal{B}=\left\{b_{1},...,b_{n}\right\}$ and switches $\mathcal{S}=\left\{s_{1},...,s_{m}\right\}$, and identify $\mathbb{R}^{n}$ with real-valued functions over $\mathcal{B}$. Then associated to $\tau$ is a linear map $L_{\tau}:\mathbb{R}^{n}\rightarrow \mathbb{R}^{m}$, and a corresponding matrix in the standard basis defined by, given $u\in \mathbb{R}^{n}$, the $j^{th}$ coordinate of $L_{\tau}(u)$ is the sum of the incoming weights, minus the sum of the outgoing weights at the $j^{th}$ switch, $1\leq j \leq m$. Let $\mathbb{R}^{n}_{+}$ denote the strictly positive orthant of $\mathbb{R}^{n}$, the collection of vectors with all positive coordinates. 

We call $L_{\tau}$ the \textit{incidence matrix} for $\tau$. Note that $\mu\in \mathbb{R}^{n}$ is a transverse measure on $\tau$ if and only if $\mu \in \ker(L_{\tau})$; thus, $\tau$ is recurrent if $\ker(L_{\tau})$ intersects $\mathbb{R}^{n}_{+}$ non-trivially. 

 As mentioned in the proof of Lemma $4.1$ of \cite{Mas-Mur},  if $\ker(L_{\tau})\cap \mathbb{R}^{n}_{+}= \emptyset$, then there is some $\delta>0$ such that  

\[ \|L_{\tau}(u)\| \geq \delta \cdot u_{min} , \hspace{1 mm} \forall \hspace{1 mm} u \in \mathbb{R}^{n}_{+}.\]

Here, $u_{min}$ is the minimum over all coordinates of the vector $u$, and $\|\cdot\|$ is the standard Euclidean norm in $\mathbb{R}^{m}$. The main goal of this section is to effectivize this statement, that is, to obtain explicit control on the size of $\delta$ as a function of $g$ and $p$: 

\begin{theorem} Let $\tau=(\mathcal{S},\mathcal{B}), |\mathcal{B}|=n, |\mathcal{S}|=m$ be a non-recurrent train track on $S_{g,p}$, and let $u\in \mathbb{R}^{n}_{+}$. Then 

\[ \|L_{\tau}(u)\|_{sup} \geq \frac{u_{min}}{12g+4p-12}, \]
where $\|\cdot\|_{sup}$ is the sup norm on $\mathbb{R}^{m}$. 

\end{theorem}

\begin{proof} 

We begin  by observing that non-recurrence is equivalent to the existence of ``extra'' branches, ones that must be assigned $0$ by any transverse measure:

\begin{lemma} Suppose that for each branch $b\in \mathcal{B}$, there is some corresponding transverse measure $\mu_{b}$ on $\tau$  such that $\mu(b)>0$. Then $\tau$ is recurrent. $\Box$

\end{lemma}

Therefore, the existence of a branch $b$ which is assigned $0$ by every transverse measure on $\tau$ is equivalent to $\tau$ being non-recurrent. We will call such a branch \textit{invisible}. 

Given $s\in \mathcal{S}$, the switch condition at $s$ represents a row vector of the matrix corresponding to the linear transformation $L_{\tau}$. This is the vector $v_{s}$ that has $1$'s in the coordinates corresponding to the incoming half branches incident to $s$, and $-1$'s in the coordinates corresponding outgoing half branches incident to $s$. Note that $v_{s}$ could also have a $\pm 2$ in place of two $1$'s, if both ends of a single branch are incident to $s$. Let $R(L_{\tau})$ denote the row space of $L_{\tau}$, the vector space spanned by the row vectors.

The following is an immediate corollary of Lemma $4.1$:

\begin{lemma} Suppose $b\in \mathcal{B}$ is an invisible branch. Then $b$ is not contained in a closed train path. $\Box$
\end{lemma} 

For $b$ a branch of $\tau$, Let $S(b) \subset \mathcal{S}$ denote the switches of $\tau$ incident to $b$; thus $|S(b)|=1$ or $2$. For $x \in S(b)$, consider the pointed universal cover $(\tilde{\tau},\tilde{x})$ with associated covering projection $\pi: (\tilde{\tau}, \tilde{x})\rightarrow (\tau, x)$.  We define $\mathcal{P}(\tilde{\tau},\tilde{x})\subseteq \tilde{\tau}$ to be the set of train paths in $\tilde{\tau}$ emanating from $\tilde{x}$ that do not traverse any branch which projects to $b$ under $\pi$.  

 Let $\tilde{\mathcal{P}} \subseteq \tilde{\tau}$ be the subset of the universal cover consisting of points contained in some train path of $\mathcal{P}(\tilde{\tau},\tilde{x})$. Any train path emanating from $\tilde{x}$ has a natural choice of orientation, by defining its initial point to be $\tilde{x}$. This induces an orientation on any branch $e$ contained in $\tilde{\mathcal{P}}$. Note that this is well-defined because $\tilde{\tau}$ does not contain closed train paths (proper or otherwise). 

We say that $\mathcal{P}(\tilde{\tau},\tilde{x})$ is \textit{unidirectional} if, whenever $e_{i},e_{j}\subseteq \tilde{\mathcal{P}}$ project to the same branch $e$ of $\tau$, the orientations of $e$ induced by $e_{i}$ and $e_{j}$ agree.

 Given $u\in \mathbb{R}^{n}$,  define the \textit{deviation} of $u$ at $s\in \mathcal{S}$, denoted by $d_{s}(u)$, to be the absolute value of the coordinate of $L_{\tau}(u)$ corresponding to $s$. It suffices to assume that, for $u$ as in the statement of the theorem, 

\begin{equation}
d_{s}(u)< \frac{u_{min}}{12g+4p-12}, \hspace{1 mm} \forall \hspace{1 mm} s\in \mathcal{S}
\end{equation}

We will use this assumption to obtain a contradiction. 

Since $\tau$ is non-recurrent, it must contain an invisible branch $b$. 

\begin{lemma} Let $s_{1},s_{2} \in S(b)$ be the two (possibly non-distinct) switches incident to the invisible branch $b$, $\tilde{s}_{1},\tilde{s}_{2} \in \tilde{\tau}$ corresponding lifts. Then at least one of $\mathcal{P}(\tilde{\tau},\tilde{s}_{i}), i=1,2$ is unidirectional. 
\end{lemma}

\begin{proof} Suppose not. Then there exist branches $(e^{i}_{j})_{i=1,2}^{j=1,2}\in \tilde{\mathcal{P}}$ such that $e^{i}_{1}, i=1,2$ project to a branch $e_{1}$ of $\tau$ with opposite orientations, and similarly for $e^{i}_{2},i=1,2$. Thus, in $\tau$ there exist two train paths starting from $s_{1}$ and ending at $e_{1}$, but which traverse $e_{1}$ in opposite directions. Concatenating these two paths produces a loop in $\tau$, which is a train path away from $s_{1}$. 

By the same exact argument, there is another loop containing the switch $s_{2}$ and the branch $e_{2}$, which is a train path away from $s_{2}$. We can then concatenate these two paths across the branch $b$ to obtain a ``dumb-bell'' shaped closed train path, which contains $b$ (see Figure $4$). This contradicts Lemma $4.2$. 
\end{proof}

\begin{figure}[H]
\centering
	\includegraphics[width=3.5in]{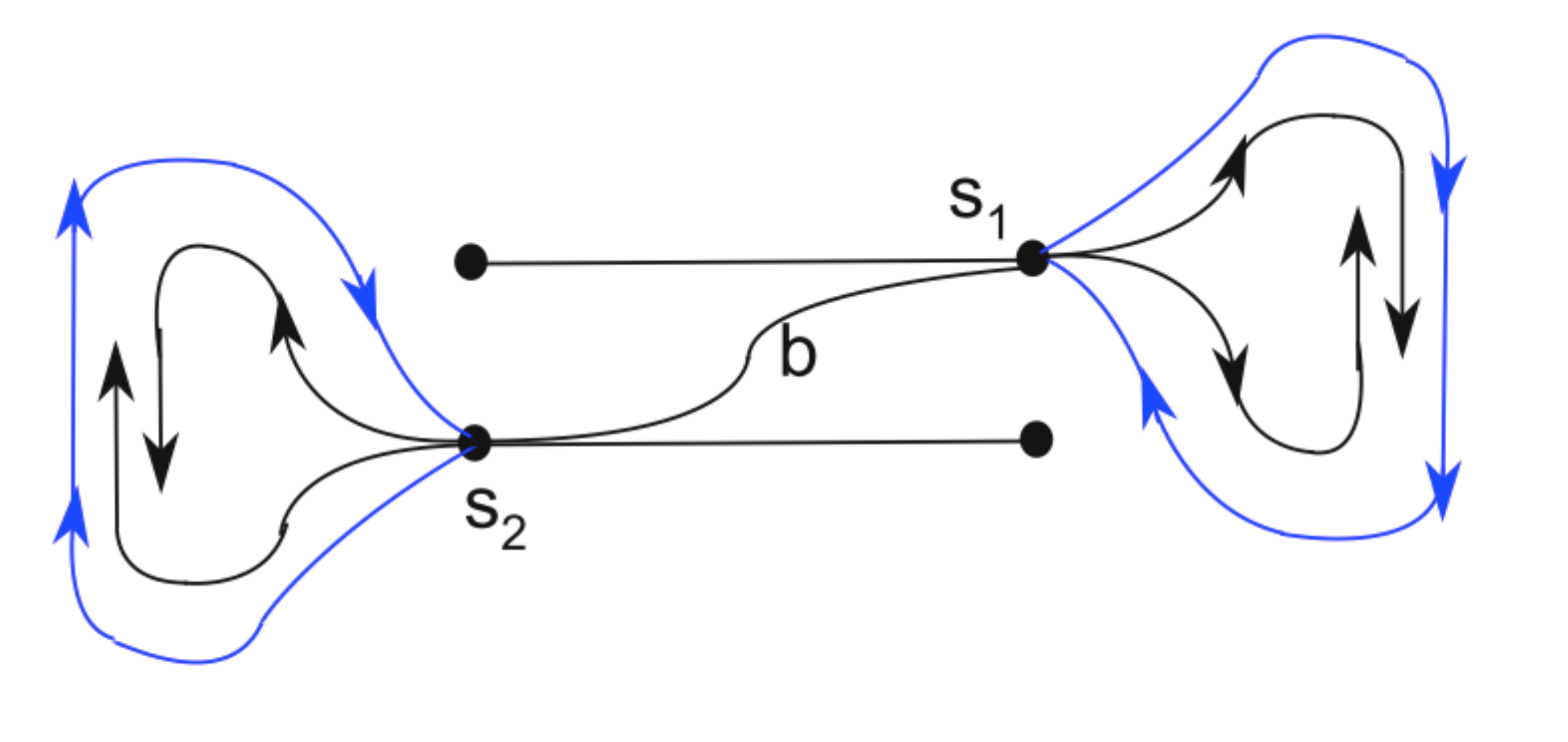}
\caption{ If neither train path set emanating from $b$ is unidirectional, then there exist non-closed train paths starting and ending at $s_{1}$ and $s_{2}$. Joining these paths across $b$ yields a closed train path containing $b$, outlined in blue above. }
\end{figure}

Therefore, we assume henceforth that $\mathcal{P}(\tilde{\tau},\tilde{s}_{1})$ is unidirectional; let $\mathcal{Q}(s_{1}) \subseteq \tau$ be the projection of $\tilde{\mathcal{P}}$ to $\tau$. That $\mathcal{P}$ is unidirectional will allow us to redefine which half branches are incoming and which are outgoing (without changing the linear algebraic structure of $L_{\tau}$) such that each branch of $\mathcal{Q}$ is mixed. 

More concretely, orient each edge $e \subseteq \mathcal{Q}(s_{1})$ by projecting the orientation on $\tilde{e}$ down to $e$, where $\tilde{e} \subseteq \tilde{\mathcal{P}}$ is any branch of $\tilde{\tau}$ with $\pi(\tilde{e})= e$; unidirectionality implies that this construction is well-defined. Then we simply define a half-branch $e'\subset e \in \mathcal{Q}$ to be outgoing at a switch $s$ if the orientation of $e'$ coming from $e$ points away from $s$, and similarly for incoming branches. Note that this is well-defined, in that two half-branches incident to the same switch in distinct germs will be assigned opposing directional classes. 

This rule then defines an assignment of direction for all half branches of $\tau$ as follows. The half branches of $\tau$ which are not contained in $\mathcal{Q}$ can be partitioned disjointly into two subcollections: the \textit{frontier} half branches (those which are incident to a switch contained in $\mathcal{Q}$), and the $\textit{interior}$ half branches (those for which the incident switch is not contained in $\mathcal{Q}$). Once directions have been assigned to the half branches of $\mathcal{Q}$ as above, directions for frontier half branches are determined by which germ they belong to at the corresponding switch. For interior half branches, simply assign the original directions coming from $\tau$. 

Let $S(\mathcal{Q})\subseteq \mathcal{S}$ denote the switches of $\tau$ contained in $\mathcal{Q}$, and recall that $v_{s}$ denotes the row vector of $L_{\tau}$ corresponding to the switch $s\in \mathcal{S}$.

\begin{lemma} The vector $V=\sum_{s\in S(\mathcal{Q})} v_{s} \in R(L_{\tau})$  is a non-zero integer vector, all of whose coordinates are non-negative. 
\end{lemma}
\begin{proof} 
Since every branch of $\mathcal{Q}$ is mixed, each component of $V$ corresponding to a branch of $\mathcal{Q}$ is $0$. The same is true for any branch not in $\mathcal{Q}$ which does not contain a frontier half-branch.

We claim that frontier half branches must be incoming at the switch contained in $S(\mathcal{Q})$ to which it is incident; this will imply that $V$ takes on a positive value for each component corresponding to a branch containing a frontier half branch. 

Indeed, let $e$ be a branch containing a frontier half branch, and let $s \in S(\mathcal{Q})$ be incident to $e$. $s\in S( \mathcal{Q} )$ implies that there is another branch $e'$ incident to $s$ such that $e'$ is a branch of $\mathcal{Q}$ and $e'$ is incoming at $s$. Thus if $e$ were outgoing at $s$, there would exist a train path emanating from $s_{1}$ which traverses $e$, by contatenating the train path starting at $s_{1}$ and ending at $e'$ with the train path connecting $e'$ to $e$ over $s$. This contradicts the assumption that $b \notin \mathcal{Q}$. 

Thus to complete the argument it suffices to show that the collection of frontier half branches is non-empty. Recall that $b$ is an invisible branch, and is therefore not contained in any closed train path. It then follows that the half branch of $b$ incident to $s_{1}$ is frontier.

\end{proof}

 We now use the following elementary fact regarding train tracks on $S_{g,p}$, (see \cite{Pen-Har} for proof):

\begin{lemma} Let $\tau\subset S_{g},\tau=(\mathcal{B},\mathcal{S})$ be a train track. Then 
\[ |\mathcal{B}|\leq 18g+6p-18;\]
\[ |\mathcal{S}|\leq 12g+4p-12. \]

\end{lemma}

Therefore, there are at most $12g+4p-12$ row vectors of $L_{\tau}$ in the sum $V$. Furthermore, since the components of $V$ are all non-negative integers, 
\[ |V \cdot u | \geq u_{min}, \]

where $\cdot$ denotes the standard Euclidean dot product. On the other hand, assuming the validity of $(4.1)$, one obtains

\[ |V \cdot u| = \left| \sum_{s\in S(\mathbb{Q})}\textbf{v}_{s} \cdot u \right|  \leq \sum_{s\in S(\mathbb{Q})} \left |\textbf{v}_{s} \cdot u \right | \]
\[  = \sum_{s \in S(\mathbb{Q})} d_{s}(u) < (12g+4p-12) \cdot \frac{u_{min}}{12g+4g-12}= u_{min}, \]

a contradiction.

\end{proof}

\section{An effective Nesting Lemma}
In this section, we will use Theorems $3.2$ and $4.3$ to establish the following effective version of Masur and Minsky's \cite{Mas-Mur} nesting lemma:

\begin{lemma}  There exists a function $k(g,p)= O(\omega^{2})$ such that if $\sigma$ and $\tau$ are large train tracks and $\sigma$ is carried by $\tau$, and $d(\tau,\sigma)>k(g,p)$, then 
\[ PN(\sigma)\subset int(PN(\tau)). \]

\end{lemma}

\begin{remark} When convenient, we will assume our train tracks to be generic; as mentioned in \cite{Mas-MurIII}, the proof of the nesting lemma in the generic case is easily extendable to the general setting.
\end{remark}

 If $\mu\in P(\tau)$, define the \textit{combinatorial length} of $\mu$ with respect to $\tau$, $l_{\tau}(\mu)$, to be the integral of $\mu$ over $\mathcal{B}$, that is 
\[ l_{\tau}(\mu):= \sum_{b}\mu(b) \]
We also define 
\[l_{N(\tau)}(\mu):= \min_{\sigma}l_{\sigma}(\mu ) \]
where the minimum is taken over all tracks $\sigma \in N(\tau)$ carrying $\mu$.

We will need the following lemma, as seen in Hammenst\"{a}dt \cite{Ham}: 

\begin{lemma} Let $c$ be a simple closed curve carried by a train track $\tau$. Then the counting measure on $c$ is a vertex cycle of $\tau$ if and only if, for any branch $b$ of $\tau$, the image of $c$ under its corresponding carrying map traverses $b$ at most twice, and never twice in the same direction. 
\end{lemma}

Since the vertex cycles are the extreme points of $P(\tau)$, by the classical Krein-Milman theorem, any projective transverse measure class can be written as a convex combination of vertex cycles; that is, given $\kappa \in P(\tau)$, there exists $(a_{i})$ such that 

 \begin{equation} \label{3.1}
\kappa= \sum_{i}a_{i}\alpha_{i} 
\end{equation}

where $(\alpha_{i})$ are the vertex cycles of $\tau$. Any train track on $S_{g,p}$ has at most $18g+6p-18$ branches, and therefore by Lemma $5.2$, if $\tau$ is any train track and $\alpha$ is a vertex cycle, 
\[ l_{\tau}(\alpha)\leq 2(18g+6p-18). \]
Lemma $5.2$ also implies that any train track $\tau$ has at most $3^{18g+6p-18}$ vertex cycles, since any branch is traversed once, twice, or no times. We therefore conclude that, given $\lambda$ as in equation \ref{3.1},   
\begin{equation}\label{3.2}
 \max_{i}a_{i}\leq l_{\tau}(\sigma)< \left[(2(18g+6p-18))\cdot 3^{18g+6p}\right] \max_{i} a_{i} 
\end{equation}

\begin{equation} \label{3.3}
= C\cdot \max_{i} a_{i} 
\end{equation}

\begin{lemma} 

Given $L>0$, there exists a function $h_{L}(g,p)=O(\log_{\omega(g,p)}(L))$ such that  if $\alpha \in P(\tau)$ and $l_{\tau}(\alpha)\leq L$, then $d_{\mathcal{C}}(\alpha,\tau)< h_{L}(g,p)$. 

\end{lemma} 

\begin{proof}

Suppose $l_{\tau}(\alpha)\leq L$. We will abuse notation and refer to the image of $\alpha$ under its carrying map by $\alpha$. Then every time $\alpha$ traverses a branch of $\tau$, by Lemma $5.2$, it can intersect a vertex cycle at most twice. Therefore, if $v$ is any vertex cycle of $\tau$, 
\[ i(v, \alpha)\leq 2L, \]
and hence by Theorem $3.2$, for any $\lambda \in (0,1)$ and $g=g(\lambda)$ sufficienly large,  
\begin{equation} \label{3.4}
d_{\mathcal{C}}(v,\alpha)\leq \frac{\log_{\omega}(2L)}{\lambda(\log_{\omega}(3)+1)-\log_{\omega}(1000\cdot \log_{2}(100\omega))}+2 
\end{equation}
\[ =O(\log_{\omega}(L)). \]

\end{proof}

\begin{remark} One needs to be cautious in manipulating the inequality in Theorem $3.2$ to obtain Equation $5.4$; if 
\[\rho(\omega,\lambda):= \lambda(\log_{\omega}(3)+1)-\log_{\omega}(1000\cdot\log_{2}(100\omega))<0,\]

 the direction of the inequality changes and we will not get the desired upper bound on curve graph distance. However, 
\[ \lim_{\omega\rightarrow \infty} \rho(\omega,\lambda)= \lambda>0,\]
and therefore for sufficiently large $\omega$ this is not an issue. 

\end{remark}

\begin{lemma} Suppose $\sigma$ is a large recurrent train track carried by $\tau$ on $S_{g,p}$, and let $\sigma' \in E(\sigma), \tau'\in E(\tau)$ such that $\sigma'$ is carried by $\tau'$. Then the total number of times, counting multiplicity, that branches of $\sigma'$ traverse any branch of $\tau'\setminus \tau$ is bounded above by $m_{0}= 36g+12p$. 
\end{lemma}
\begin{proof} The complete argument can be found in Masur and Minsky's original paper \cite{Mas-Mur} on the hyperbolicity of the curve complex. For our purposes and for the sake of brevity, it suffices here to simply remark that they show any given branch of $\sigma'$ can only traverse branches of $\tau'\setminus \tau$ at most twice. Then, since any track has less than $18g+6p$ branches, the result follows. 
\end{proof}

To prove the following lemma, we use the results from section $4$:

\begin{lemma} There exists $R=R(g,p)$ with 
\[ \frac{1}{R(g,p)}= O\left(\omega^{2}\right),\]

 such that if $\sigma<\tau$ and $\sigma$ is large and $\tau$ is generic, $\mu\in P(\tau)$ and every branch $b$ of $\tau\setminus \sigma$ and $b'$ of $\sigma$ satisfies $\mu(b)<R(g)\mu(b')$, then $\mu\in int(PE(\sigma))$, and $\sigma$ is recurrent.
\end{lemma}

\begin{proof}
We follow Masur and Minsky's original argument \cite{Mas-Mur}. The main tools are the elementary moves on train tracks called \textit{splitting} and \textit{shifting} as introduced in section $3$ (see Figures $2$ and $3$), which can be used to take $\tau$ to a diagonal extension of $\sigma$. In order to do this, we need to move any branch of $\tau\setminus \sigma$ into a corner of a complementary region of $\sigma$. A split or a shift applied to any such branch either reduces the number of branches of $\tau\setminus \sigma$ incident to a given branch of $\sigma$, or decreases the distance between a branch of $\tau \setminus \sigma$ and a corner of a complementary region of $\sigma$. 

Thus, a bounded number of such moves produces a track carried by a diagonal extension of $\sigma$.  If a splitting is performed involving a branch $b$ of $\tau\setminus \sigma$ and a branch $c$ of $\sigma$, the resulting track contains a new branch $c'$ of $\sigma$, and we can extend $\mu$ to $c'$ to be consistent with the switch conditions by assigning $\mu(c')= \mu(c)-\mu(b)$. In particular, a sufficient condition for being able to define $\mu$ on the new track is

\begin{equation}\label{4.5}
 \mu(c)>\mu(b). 
\end{equation}

There are at most $18g+6p$ branches of $\tau\setminus \sigma$, and at most $18g+6p$ branches of $\sigma$ or $\tau$. As earlier mentioned, a splitting move either reduces the number of branches of $\tau\setminus \sigma$ incident to $\sigma$, or it reduces the number of edges of $\sigma$ between a given branch of $\tau \setminus \sigma$ and a corner that it faces. Once a branch of $\tau \setminus \sigma$ is separated by a corner of  a complementary region of $\sigma$ by only edges of $\sigma$ for which no splitting moves can be performed, a shift move takes such an edge to a corner point. Therefore, each edge of $\tau \setminus \sigma$ is taken to a corner of $\sigma$ after no more than $18g+6p+1$ shiftings and splittings, and therefore we obtain $\tau'$ after at most $(18g+6p)(18g+6p+1)$ such moves. 

Now, let $R(g,p)= \frac{1}{(18g+6p)(18g+6p+1)+1}$, and assume that for this value of $R$, the hypothesis of the statement is satisfied. In light of equation \ref{4.5}, $\mu$ is definable on the diagonal extension $\tau'$ we obtain after splitting and shifting as long as 

\begin{equation} \label{4.6}
\min_{\sigma}\mu > \frac{1}{R(g,p)}\max_{\tau\setminus \sigma} \mu, 
\end{equation}

which is precisely what the hypothesis of Lemma $5.6$ implies. Therefore, $\mu$ is extendable to a diagonal extension of $\sigma$ such that all branches receive positive weights, hence $\mu\in int(PE(\sigma))$. 

It remains to show that $\sigma$ is recurrent; suppose not. Let $\mathcal{B}(\sigma)$ denote the branch set of $\sigma$. Then Theorem $4.3$ implies that if $u\in \mathbb{R}^{|\mathcal{B}(\sigma)|}$ is a vector with all positive coordinates, 

\[ \|L_{\sigma}(u)\| \geq \frac{u_{min}}{12g+4p-12}. \]

In light of equation $5.6$, the vector $\mu$ has small deviations, since $\mu$ satisfies the switch conditions on $\sigma$, up to the additive error coming from the weight it assigns to any branch of $\tau\setminus \sigma$, which is less than 

\[ \frac{\mu_{min}}{R(g,p)}; \]

since we assumed that $\tau$ is generic, there are at most two branches of $\tau \setminus \sigma$ incident to any branch of $\sigma$, and therefore the deviations of $\mu$ are all less than $\frac{\mu_{min}}{12g+4p-12}$, contradicting Theorem $4.3$. 

\end{proof}

\begin{lemma} Let $L>0$ be given. Then there exist functions $s_{L}(g,p)$ and $y(g,p)=O(\omega^{3}3^{18\omega})$ satisfying the following:
If $\sigma$ is large and carried by $\tau$ and $\sigma'\in E(\sigma), \tau'\in E(\tau)$ such that $\tau'$ carries $\sigma'$, and if $d_{\mathcal{C}}(\sigma,\tau)\geq s_{L}$, then any simple closed curve $\beta$ carried on $\sigma'$ can be written in $P(\tau')$ as $\beta_{\tau}+\beta'_{\tau}$, and such that
\[ l_{\tau'}(\beta'_{\tau})\leq y(g,p)\cdot l_{\sigma'}(\beta), \hspace{1 mm} \mbox{and}\]
\[ l_{\tau}(\beta_{\tau})\geq s_{L}(g,p) l_{\sigma'}(\beta). \]

\end{lemma}

\begin{proof} The details of the argument are not entirely relevant for the proof of our main theorem, and can be found in Masur-Minsky \cite{Mas-Mur}; therefore we omit the particulars of the proof, and remark only that in their argument, Masur and Minsky show that it suffices to take 
\[ y(g,p):= C\cdot m_{0}W_{0}C_{0},\]
where $C$ is the constant from equation \ref{3.3}, $m_{0}$ is the constant from the statement of Lemma $5.5$, $W_{0}$ is a bound on the weights that a vertex cycle can place on any one branch of $\sigma'$ (and therefore it suffices to take $W_{0}=3$ by Lemma $5.2$), and $C_{0}$ is a bound on the combinatorial length of any vertex cycle on any train track on $S_{g,p}$. Putting all of this together, we obtain 
\[ y(g,p):=  \left[(2(18g+6p-18))\cdot 3^{18g+6p}\right] (3(36g+12p-36)^{2}) \]
\[= O(\omega^{3}3^{18\omega}), \]
as claimed. 

They also show that it suffices to take 
\[ s_{L}(g,p):= h_{L}(C_{0}L+y(g,p))+2B, \]
where $B$ is a bound on the curve graph distance between any two vertex cycles of the same train track.  

Therefore, by Theorem $3.1$, for sufficiently large $\omega$, 

\begin{equation}\label{4.6}
s_{L}(g,p)\leq h_{L}(C_{0}L+y(g,p))+6. 
\end{equation}

\end{proof}

\textit{Proof of Lemma $5.1$}.

Again with concision in mind, we do not include the entirety of Masur and Minsky's argument; we simply remark here that in our notation, it suffices to choose 
\[ k(g,p):= s_{Cm_{0}\cdot \left(\frac{m_{2}}{R(g,p)}\right)^{m_{3}}}(\omega) \]

Here, $m_{0}$ is as in Lemma $5.5$ and is thus bounded above by $36g+12p$, $m_{2}< (18g+6p)^{18g+6p}$, and $m_{3}<18g+6p$. Thus 
\[ Cm_{0} \cdot \left(\frac{m_{2}}{R(g)}\right)^{m_{3}}\]
\[< \left[(2(18g+6p-18))\cdot 3^{18g}\right] \cdot  (36g+12p)\left( (18g+6p)^{18g+6p+2}  \right)^{18g+6p}=:D,\]

and therefore by Lemma $4.5$, for $\omega(g,p)$ sufficiently large, 

\[ k(g,p)< h_{D}(C_{0}D+y(g,p))+6\]
\[= O(\log_{\omega}(\omega^{3}3^{18\omega}(18\omega)^{324\omega^{2}+36\omega}))\]
\[ =O(\omega^{2}).  \Box \]

\section{Proof of the main theorem and corollaries}
In this section, we prove the main results:

\textbf{Theorem 1.1}: \textit{ Let $\omega(g,p)= 3g+p-4$. There exists a function $K(g,p)=O(\omega(g,p)^{2})$ such that any nested train track sequence with $R$-bounded steps is a $(K(g,p)+R)$-unparameterized quasi-geodesic of the curve graph $\mathcal{C}_{1}(S_{g,p})$, which is $(K(g,p)+R)$-quasiconvex. } \vspace{5 mm}

\textit{Proof}: Where possible, we use the same notation that Masur and Minsky do to avoid confusion. Let $\delta$ be the hyperbolicity constant of $\mathcal{C}_{1}(S)$. By Hensel-Przytycky-Webb \cite{Hen-Prz-We}, it suffices to take $\delta=17$. Let $B$ be a bound on the diameter of the set of vertex cycles of a given train track $\tau \subset S_{g,p}$. As mentioned above, for sufficiently large $\omega$ it suffices to take $B=3$ (see \cite{Aoug} for a proof of this).

Given a nested train sequence $(\tau_{i})_{i}$, consider a subsequence $(\tau_{i_{j}})_{j}$ such that
\[ k(g,p) \leq d_{T}(\tau_{i_{j}},\tau_{i_{j+1}})< k(g,p)+R, \]
and such that if $\tau_{n}$ is any track not in the subsequence $(\tau_{i_{j}})_{j}$, then there is some $c$ for which 
\[ d_{T}(\tau_{i_{c}},\tau_{n})<k(g,p). \]

Then since $d_{T}(\tau_{i_{j}},\tau_{i_{j+1}})\geq k(g,p)$, the effective nesting lemma implies that 
\[ PN(\tau_{i_{j+1}})\subset int(PN(\tau_{i_{j}})) \]

For any train track $\tau$, one always has

\[ \mathcal{N}_{1}(int(PN(\tau)))\subset PN(\tau),\]
 where $\mathcal{N}_{m}$ denotes the $m$-neighborhood in $\mathcal{C}_{1}$. Combining these two inclusions and inducting yields 
\[ \mathcal{N}_{m-1}(PN(\tau_{i_{j+m}}))\subset int(PN(\tau_{i_{j}})). \]

Masur and Minsky then make use of a lemma which implies that no vertex of $\tau_{i_{j}}$ is in $int(PN(\tau_{i_{j}}))$, and therefore

\[ d_{T}(\tau_{i_{j}},\tau_{i_{k}})\geq |k-j|. \]

Thus if $(v_{i_{j}})_{j}$ is any sequence of the vertices of $(\tau_{i_{j}})_{j}$, we have 

\[ |m-n| \leq d_{\mathcal{C}}(v_{i_{n}},v_{i_{m}})< (k(g,p)+R+2B)|m-n|, \]
which implies that $(v_{i_{j}})_{j}$ is a $(k(g,p)+R+2B)$-quasigeodesic. This proves the first part of Theorem $1.1$, with $K(g,p):=2k(g,p)+46$ (we've shown the sequence to be a $(k(g,p)+R+6)$-quasi-geodesic, but we will need the extra $k(g,p)+40$ for the quasiconvexity statement).

We now show $(\tau_{i})_{i\in I_{1}}$ is $(K(g,p)+R)$-quasiconvex.  In any $\delta$-hyperbolic metric space, a geodesic segment $\gamma$ connecting the endpoints of a $K$-quasigeodesic segment $\gamma'$ is contained in a $W$-neighborhood of $\gamma'$, where $ W=W(K,\delta)$. $W$ is sometimes known as the \textit{stability constant}.

Therefore, a geodesic segment connecting any two elements of the vertex cycle sequence $(v_{i_{j}})_{j}$ is contained in a $W(K,\delta)=W(k(g,p)+R+6,17)$-neighborhood of the sequence. 

\begin{lemma} For sufficiently large $\omega$, $W< K(g,p)+R$. 
\end{lemma}

\begin{proof} We only give a sketch here; the main idea of the proof follows an argument on page $35$ of Ohshika \cite{Ohs}, and we refer to this for a more complete argument. Hyperbolicity of $\mathcal{C}_{1}$ implies the existence of an exponential divergence function; that is, if $\alpha_{1},\alpha_{2}:[0,\infty) \rightarrow \mathcal{C}_{1}$ are two geodesic rays based at the same point $x_{0}\in \mathcal{C}_{1}$, then there is some exponential function $f$ so that for suficiently large $r$ (depending on the choice of geodesic rays), the length of any arc outside of a ball of radius $r$ centered at $x$, connecting $\alpha_{1}(r)$ and $\alpha_{2}(r)$ is at least $f(r)$. 

Let $x,y$ be two elements of a vertex cycle sequence $(v_{i_{j}})_{j}$, and let $h$ be a geodesic segment connecting them. Denote by $w$ the $(k(g,p)+M+6)$-quasigeodesic segment obtained by following along the vertex sequence from $x$ to $y$. 

Let $D=\sup_{x\in h}d_{\mathcal{C}}(x,w)$, and suppose $s\in h$ with $d_{\mathcal{C}}(s,w)=D$. Let $a$ and $b$ be two points on $w$ whose distance from $s$ is $D$ and such that $a$ and $b$ are on different sides of $s$. Note that we can assume that such points exist, because the end points of $w$ are also the endpoints of $h$, and therefore $s$ must be at least $D$ from the end points of $w$. 

Let $a'$ (resp $b'$) be points located $2D$ from $s$ on either side of $s$ on $w$; if $s$ is closer than $2D$ to one of the endpoints of $w$, simply define $a'$ (resp. $b'$) to be this corresponding endpoint of $w$. Let $y,z\in h$ be points whose distances are less than $D$ from $a',b'$ respectively. Note that there is an arc $\sigma$ joining $y$ to $z$, by first connecting $y$ to $a'$, then $a'$ to $b'$ along $w$, and then jumping back over to $h$. Thus

\[ d_{\mathcal{C}}(y,z)\leq d_{\mathcal{C}}(y,a')+d_{\mathcal{C}}(a',b')+d_{\mathcal{C}}(b',z)\]
\[ \leq D+4D+D= 6D. \]

This gives a bound on the length of the segment of $w$ connecting $y$ and $z$ since it is a quasi-geodesic:

\[ \mbox{length}_{w}(y,z)\leq (k(g,p)+R+6)\cdot 6D. \]

Let $\beta$ be the arc obtained by concatenating the following $5$ arcs: the arc along $h$ from $a$ to $a'$, the arc connecting $a'$ to $y$, the arc along $w$ from $y$ to $z$, the arc connecting $z$ to $b'$, and the arc along $h$ from $b'$ to $b$ (see Figure $5$). 

It follows that 
\[ \mbox{length}(\beta)\leq 4D+ (k(g,p)+R+6)D.\]

Now we use the divergence function $f$ for $\mathcal{C}_{1}$ to bound the length of $\beta$ from below. Indeed, for sufficiently large $D$, we have 

\[ \mbox{length}(\beta)\geq f(D- c),\]

where $c$ is a constant related related to $f(0)$, and which does not affect the growth rate of the function $f$. Therefore, 

\[ f(D-c)\leq 4D+(k(g,p)+R+6)D.\]

\begin{figure}
\centering
	\includegraphics[width=3.5in]{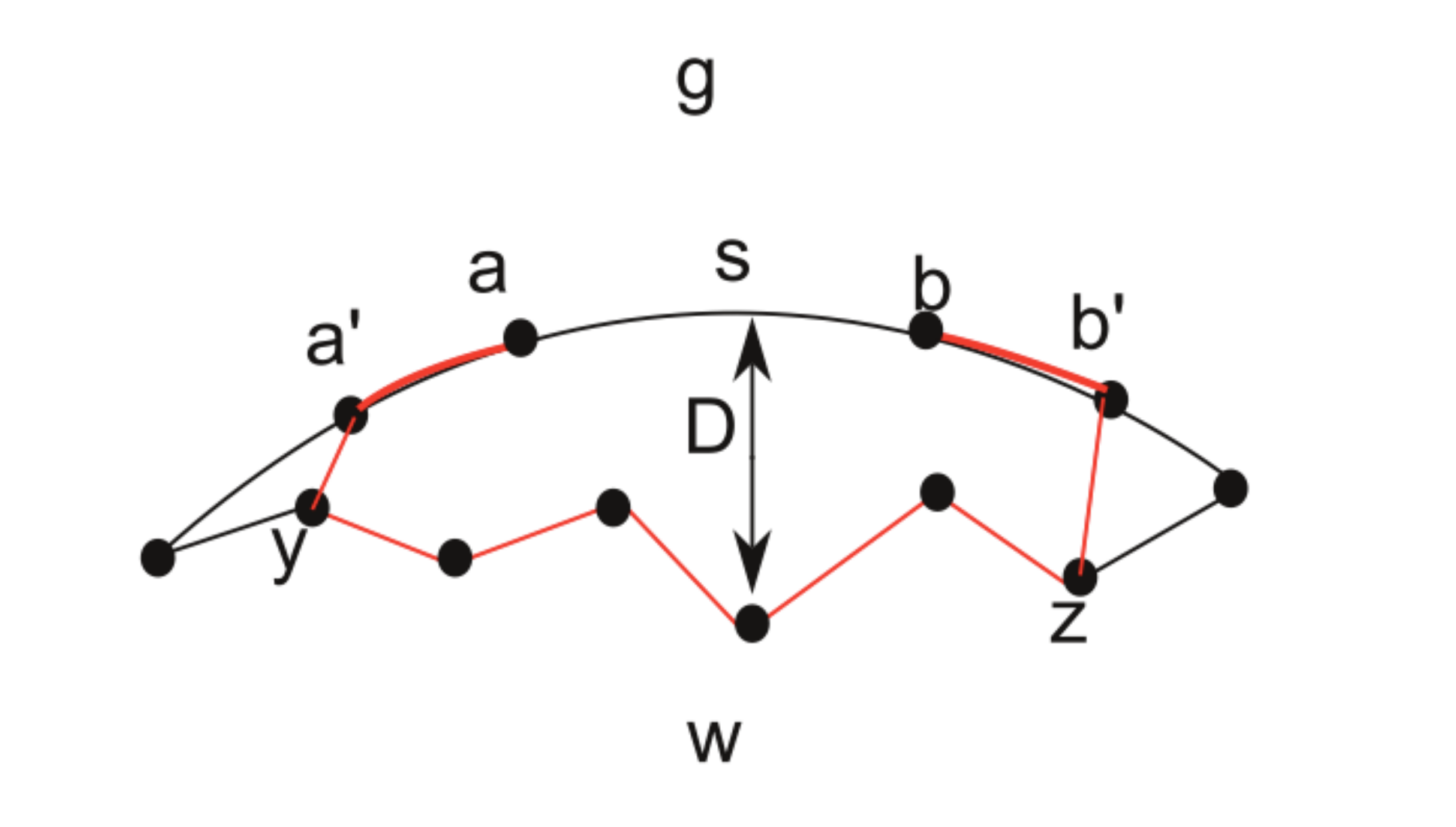}
\caption{ The length of the path $\beta$ outlined in red is bounded above by $4D+(k(g,p)+R+6)D$. }
\end{figure}

Therefore, if $D> k(g,p)+R+6$, $\omega$ can not be arbitrarily large because $f(x)$ eventually dominates $x^{2}$. 
\end{proof}

\begin{remark} We note that the conclusion of Lemma $6.1$ is not at all sharp; indeed, the same argument would have shown that $W$ is eventually smaller than $(k(g,p)+R+6)^{\lambda}$ for any $\lambda\in (0,1)$. However we do not concern ourselves with this, because the contribution to the quasiconvexity of nested sequences coming from $W$ will be dominated by a larger term, as will be seen below. 

\end{remark}

We have now shown that the collection of vertices of the sequence $(\tau_{i_{j}})_{j}$ is quasiconvex with quasi-convexity constant $k(g,p)+R+6$. It remains to analyze the vertex cycles of tracks that are not in this subsequence. If $v$ is such a vertex and $\omega$ is sufficiently large, we know that $v$ is within $k(g,p)+6$ from some vertex of one of the $\tau_{i_{j}}$'s. In any $\delta$-hyperbolic space, geodesics with nearby end points \textit{fellow travel}, in that they remain within a bounded neighborhood of one another, whose diameter depends only on $\delta$ and the distance between endpoints. 

Indeed, if $h$ is any geodesic segment connecting arbitrary vertices $v_{1},v_{2}$, $h$ must remain within $2\delta+k(g,p)+6 \leq 40+k(g,p)$ of some geodesic connecting vertices of the $\tau_{i_{j}}$. 

Therefore, the collection of all vertices of the sequence $(\tau_{i})_{i\in I_{1}}$ is a $(46+R+2k(g,p))$-quasiconvex subset of $\mathcal{C}_{1}$.  $\Box$

\subsection{Proof of Corollary $1.3$}

Masur and Minsky complete their argument showing the quasiconvexity of $D(g)\subset \mathcal{C}_{1}(S_{g})$ by noting that any two disks in $D(g)$ can be connected by a path in $D(g)$ representing a \textit{well-nested curve replacement sequence}, a certain kind of nested train track sequence with $R$-bounded steps for which one can take $R$ to be $15$. 

Thus we see that $D(g)$ is $(61+4k(g,0))$-quasiconvex, and this completes the proof of Corollay $1.3$.   $\Box$

\subsection{Proof of Theorem $1.2$}
The purpose of this subsection is to prove Theorem $1.2$, which states that the splitting and sliding sequences project to $O(\omega^{2})$-unparameterized quasi-geodesics in the curve graph of any essential subsurface $Y\subseteq S$. To do this, we simply follow the original argument of Masur-Mosher-Schleimer \cite{Mas-Mos-Sch}, effectivizing along the way. 

We first introduce some terminology; given a subsurface $Y$, as in section $2$, let $S^{Y}$ denote the (non-compact) covering space of $S$ corresponding to $Y$. Then if $\tau$ is a train track on $S$, let $\tau^{Y}$ denote the pre-image under the covering projection of $\tau$ to $S^{Y}$. Then let $\mathcal{C}(\tau^{Y})$ and $\mathcal{AC}(\tau^{Y})$ denote the collection of essential, non-peripheral, simple closed curves (respectively curves and arcs) in the Gromov compactification of $S^{Y}$ whose interiors are train paths on $\tau^{Y}$. Let $V(\tau)$ denote the collection of vertex cycles of a track $\tau$. 

Then if $Y$ is not an annulus, define the \textit{induced track}, denoted $\tau|Y$, to be the union of branches of $\tau^{Y}$ traversed by some element of $\mathcal{C}(\tau^{Y})$. 

 We first note that any splitting and sliding sequence $(\tau_{i})_{i}$ is a nested train track sequence with $Z$-bounded steps, for $Z$ some uniform constant. Indeed, if $\tau_{i}$ is obtained from $\tau_{i-1}$ by either a splitting or a sliding, any vertex cycle of $\tau_{i}$ may intersect a vetex cycle of $\tau_{i-1}$ at most $6$ times over any branch of $\tau_{i-1}$. Thus there is some linear function $f: \mathbb{N} \rightarrow \mathbb{N}$ such that $i(v_{i},v_{i-1})<f(\omega(g,p))$ for $(\tau_{i})_{i}$ a sliding and splitting sequence on $S_{g,p}$, and $v_{i}$ (resp. $v_{i-1}$) is any vertex cycle of $\tau_{i}$ (resp. $\tau_{i-1}$), and therefore as a consequence of Theorem $3.2$, for sufficiently large $\omega$,
\[ d_{\mathcal{C}}(v_{i},v_{i-1}) < 4. \] 
 To show that $(\psi_{Y}(\tau_{i}))_{i}$ is a $O(\omega^{2})$-unparameterized quasi-geodesic in $\mathcal{C}(Y)$, we will exhibit a splitting and sliding sequence $(\sigma_{i})_{i}$ on $Y$ such that $d_{\mathcal{C}}(\tau_{i},\sigma_{i})= O(1)$. Then we'll be done by applying Theorem $1.1$ to the sequence $(\sigma_{i})$.

 Given a vertex cycle $\alpha$ of $\tau_{j}|Y$, define $\sigma_{j}\subset \tau_{j}|Y$ to be the minimal track carrying $\alpha$; thus $\sigma_{j}$ is recurrent by construction, and Masur, Mosher and Schleimer show $\sigma_{j}$ to be transversely recurrent as well. 

Furthermore, they show that $\sigma_{j+1}$ is obtained from $\sigma_{j}$ by a slide or a split, so long as $\sigma_{j}\neq \sigma_{j+1}$. Therefore $(\sigma_{i})_{i}$ constitutes a sliding and splitting sequence of birecurrent train tracks, and thus is a nested train track sequence on $Y$ with $Z$- bounded steps.

Since $\sigma_{j}$ is a subtrack of $\tau_{j}|Y$, by Lemma $5.2$, any vertex cycle of $\sigma_{j}$ is a vertex cycle of $\tau_{j}|Y$, and therefore the diameter of $V(\tau_{j}|Y)\cup V(\sigma_{j})$ is no more than $6$ for sufficiently large $\omega$. 

Since $\alpha$ is carried by $\tau_{j}|Y$, it is also carried by $\tau_{j}$. Masur, Mosher, and Schleimer then make use of a lemma which implies the existence of a vertex cycle $\beta_{j}$ of $\tau_{j}$ which intersects the subsurface $Y$ essentially. By Lemmas $2.8$ and $5.4$ of \cite{Mas-Mos-Sch}, 
\[ i(\pi_{Y}(\beta_{j}), v_{j}) < 8|\mathcal{B}(\tau_{j})|, \] 
and therefore by Lemma $4.7$ and Theorem $3.2$, for $\omega$ sufficiently large,
\[ d_{\mathcal{C}}(\pi_{Y}(\beta_{j}),v_{j} )< 4 .\]
 This same argument applies to any vertex cycle of $\tau_{j}$ which projects non-trivially to $Y$, and thus we conclude that 
\[ d_{Y}(\sigma_{j}, \tau_{j}) \leq d_{Y}(\sigma_{j}, \tau_{j}|Y)+ d_{Y}(\tau_{j}|Y, \tau_{j})\]
\[ < 6 + 4 =10, \]
for all $\omega$ sufficiently large. $\Box$

\end{document}